\documentclass{article}
\usepackage{amssymb}
\usepackage{dsfont}
\usepackage{textcomp}
\usepackage{amsmath}
\usepackage{pgf}
\usepackage{txfonts}
\usepackage[latin1]{inputenc}
\usepackage[T1]{fontenc}
\usepackage{graphicx}

\newtheorem{theorem}{Theorem}[section]
\newtheorem{corollary}[theorem]{Corollary}

\newtheorem{proposition}[theorem]{Proposition}

\newtheorem{theo}{Theorem}[section]
\newtheorem{defi}[theo]{Definition}
\newenvironment{definition}{\begin{defi}\rm}{\end{defi}}
\newtheorem{remarque}[theo]{Remark}
\newenvironment{remark}{\begin{remarque}\rm}{\end{remarque}}
\newtheorem{exemple}[theo]{Example}
\newenvironment{example}{\begin{exemple}\rm}{\end{exemple}}
\newtheorem{prf}{\it{Proof.}}

\newenvironment{proof}{\begin{prf}\rm}{\hfill$\Box$\end{prf}}

\def\N{\mathbb{ N}}
\def\Z{\mathbb{ Z}}
\def\Q{\mathbb{ Q}}
\def\R{\mathbb{ R}}
\def\C{\mathbb{ C}}

\title{On generalized series fields and exponential-logarithmic series fields with derivations.}

\author{Micka\"{e}l Matusinski \thanks{The author thanks Salma Kuhlmann for reading a preliminary version of this paper and providing helpful comments. He thanks also the referee of the journal Order for having corrected several mistakes and provided interesting comments, in particular for having suggested the statement of Theorem 2.5.}}

\begin{document}

\maketitle

\begin{abstract}
We survey some important properties of fields of generalized series and of exponential-logarithmic series, with particular emphasis on their possible differential structure, based on a joint work of the author with S. Kuhlmann \cite{matu-kuhlm:hardy-deriv-gener-series,matu-kuhlm:hardy-deriv-EL-series}.
\end{abstract}

%


\section{Introduction on generalized series fields}

\begin{definition}[Generalized series and their natural valuation]
Let $k$ be a  field  and $\Gamma$ be a totally ordered Abelian group. A \textbf{generalized series} with coefficients in $k$ and exponents in  $\Gamma$ is a map $a:\Gamma\rightarrow k$, denoted by :
\begin{center}
$a=\displaystyle\sum_{\alpha\in\Gamma}a_\alpha t^\alpha$
\end{center}
 with its support $\mathrm{Supp}\ a:=\{\alpha\in\Gamma\ |\ a_\alpha\neq  0\}$ that is well-ordered. 

We denote by $k((\Gamma))$ the set of generalized series, which is actually a \textbf{field} when endowed with the componentwise sum and (the straightforward generalization of) the convolution product. 

The \textbf{canonical valuation} on $k((\Gamma))$ is:
\begin{center}
$\begin{array}{lccc}
v:&k((\Gamma))&\rightarrow&\Gamma\cup\{\infty\}\\
&a&\mapsto&\min(\mathrm{Supp}\ a)\\
&0&\mapsto&\infty
\end{array}$
\end{center}
Since $\Gamma$ is an ordered group, one can define on it the \textbf{Archimedian equivalence relation} :
\begin{center}
	$\forall \alpha_1,\alpha_2\in\Gamma,\ \alpha_1 \sim_+ \alpha_2 \Leftrightarrow \exists n,\ n|\alpha_1|\leq |\alpha_2|\ \textrm{and}\ n|\alpha_2|\leq |\alpha_1|$
\end{center}
and the corresponding set $\Phi$ of \textbf{Archimedian equivalence classes} $\phi:=[\gamma]$ for $\gamma\in\Gamma$. 
The set $\Phi$ inherits naturally a total order $\leq$ from the one of $\Gamma$: \begin{center}
	 $\phi_1=[\alpha_1]\leq \phi_2=[\alpha_2]\leq \infty:=[0]\Leftrightarrow 0\leq |\alpha_2|\leq |\alpha_1|$.
\end{center} The order type of $\Phi$ is called the \textbf{rank} of $\Gamma$, and as well the \textbf{rank} of $k((\Gamma))$.
\end{definition}

Generalized series are ubiquitous objects in mathematics. They generalize the classical field of Laurent series ($\Gamma=\Z$), and Puiseux series form a subfield of $k((\Gamma))$  for $\Gamma=\Q$. 

 A few years after the pioneering work of Veronese and Levi-Civita on infinite and infinitesimal numbers \cite{veronese:fundamenti-geom, levi-civita:infinitesimi,levi-civita:numeri-transfiniti}, H. Hahn introduced in \cite{hahn:nichtarchim} the generalized series in the case where $k=\R$, along the proof of his embedding theorem for (totally) ordered Abelian groups:
\begin{theorem}[Hahn's embedding]\label{theo:hahn1}
Any ordered Abelian group is order isomorphic to a subgroup of the Hahn group over the divisible hull of its skeleton (see Section \ref{sect:deriv}). 
\end{theorem}
 To quote Hahn himself (p. 614-615, translated by ourselves):\\
`` \emph{We want to show that to the sizes in an arbitrary non Archimedian system of orders of magnitude, in unambiguous way symbols can be assigned of the form
\begin{center}
$\displaystyle\sum a_\gamma e_\gamma$
\end{center}
where $a_\gamma$ means a real number, $e_\gamma$ are symbols for a ranking (``units") and summation applies to well-ordered amounts.}"\\
For a survey on this topic, we recommend \cite{ehrlich:hahn-nichtarchim}.

Almost at the same time, G.H. Hardy published his monograph enhancing P. Du Bois-Reymond's work on \emph{orders of infinity} \cite{hardy}. The central question that this work addresses is: even though orders of growth of real functions have no bound (by Du Bois-Reymond's results), can we find asymptotic scales to describe them ? With the field of \emph{logarithmic-exponential functions}, Hardy provides an important example of such a generalized asymptotic scale. It is what is called now a \emph{Hardy field} (see \cite[Section V, Appendix]{bour:unevar} and below). 

In his seminal article \cite{krull:allg-bewertungstheorie}, W. Krull provides a common framework for these works, and many others: \emph{abstract valuation theory}. Recall that, given a valued field, an \textbf{immediate extension} of it is an extension of valued field with same valued group and residue field. In \cite[Section 13]{krull:allg-bewertungstheorie}, fields of generalized series are key examples of \textbf{maximally valued fields} - i.e. fields which admit no proper immediate extension - with any given value group and residue field. Note that Krull proves in the same article that \emph{any} valued field possesses an immediate maximally valued extension, and puts the following crucial question (see the end of Section 13): \\
``\emph{Under which conditions a maximal extension of a given valued field may be viewed as a generalized series field?}"

The answer for valued fields having same characteristic as its residue field is given by Kaplansky in \cite{kap,kap2} almost ten years after. Recall that two valued fields are said to be \textbf{analytically isomorphic} if there exists a value preserving isomorphism between them. More precisely:

\begin{theorem}[Kaplansky's embedding]\label{theo:kap}
In the equal characteristic case, every valued field is analytically isomorphic to a subfield of a suitable generalized series field
\end{theorem}
``Suitable" means that in certain cases, one might need \emph{factor sets} to express the multiplication rule for monomials of the generalized series. But in the context we are interested in (real closed or algebraically closed zero characteristic residue field), these are not needed.

Note that there is also a result for the mixed characteristic case, that we leave aside since we are mostly interested in the zero characteristic for differential structures. Just keep in mind that generalized series are not less represented in number theory  (see e.g. \cite{schilling:arithm-gen-series, poonen:max-complete-fields, kedlaya_alg-cl-series-posit-char, kedlaya:gen-series-p-adic-alg-closure,
mourgues-ressayre_rcf-integ-part,berarducci:factor-gen-series})

For the same reason, we will not detail the various results for what are called now \emph{Mal'cev-Neumann series} \cite{neumann:ord-div-rings, malcev:embed-gr-algebras} (when $k$ is only supposed to be a ring, and $\Gamma$ is not necessarily commutative)  (see e.g. \cite{rib:properties-gen-series, rib:neumann-regular-gen-series}).\\

As for the classical formal power series, generalized series have great algebraic properties \cite{rib:series-fields-alg-closed}. E.g., they are always Henselian valued fields. If $\Gamma$ is divisible and $k$ is real closed,  respectively algebraically closed, then so is $k((\Gamma))$. As wrote S. Maclane \cite{maclane:universal-series}: \emph{generalized power series are universal as valued fields}.

In topological terms, a valued field $(K,v)$ is an \emph{ultrametric} space. In this context, generalized series fields are \textbf{spherically complete}: the intersection of any decreasing sequence of balls is nonempty. This holds if and only if every \emph{pseudo-Cauchy sequence} in $K$ has a pseudo limit in $K$, which means that the field $K$ is maximally valued. 
In a spherically complete space, many of the classical results of functional analysis hold, as the Hahn-Banach, Banach-Steinhaus and Open Mapping Theorems \cite{schneider:nonarchim-funct-analysis}. In particular, one has a \emph{Banach Fixed Point Theorem} \cite{p-crampe-rib_FP-gen-series, kuhl:Hensel-lemma}. \\

At the interplay between model theory and geometry sits the  \emph{order minimal} or \emph{o-minimal} geometry, a generalization of \emph{semialgebraic} and \emph{subanalytic} geometry \cite{vdd:tame}. There, generalized series provide non standard models for various o-minimal theories, like certain \emph{expansions} of the field of real numbers: by the power functions corresponding to a subfield of the reals \cite{miller:power}; by the global real exponential function possibly with restricted analytic functions \cite{wil:modelcomp,dmm:exp} ;  by solutions of Pfaffian differential equations \cite{speiss:clos}. Concerning the o-minimal theory with the global exponential function, there are several similar constructions of non standard models of it based on generalized series: see Section \ref{sect:EL-series} and \cite{vdd:LE-pow-series, kuhl:ord-exp, vdh:transs_diff_alg} for details. Moreover, generalized series are themselves used to build other o-minimal structures, by expanding the field of reals by \emph{multisummable series} \cite{vdd:speiss:multisum}, by \emph{convergent generalized series} \cite{vdd:speiss:gen}.

O-minimality tells us that the definable subsets of the structure have finitely many connected components, justifying the \emph{tameness} of such geometry in the sense of \cite{grothendieck:esquisse}. Other important tameness results about differential equations may be put in parallel: the proof of the Dulac conjecture (finiteness of the number of limit cycles of a planar polynomial vector field: Hilbert 16th problem, part 2) using \emph{transseries} (see \cite{ecalle:dulac, vdh:transs_diff_alg} and Section \ref{sect:EL-series} below); the desingularization result for 3-dimensional real analytic vector fields in \cite{cano_moussu_rolin} using rank 1 generalized series.\\

In o-minimal geometry as for differential equations, another important algebraic object is omnipresent: the \textbf{Hardy fields}, i.e. fields of germs at $+\infty$ of real functions closed under differentiation. As the germs in such a field are non oscillating, they carry a total order and therefore the corresponding \emph{natural convex valuation}  (the valuation ring is the convex hull of the constant functions), so they are \textbf{valued differential fields} i.e. fields equipped with a valuation and a derivation. Moreover, there are specific connections between the derivation and the valuation: the valuation is \textbf{differential with a certain principle (*)} \cite[p.303 and 314]{rosenlicht:diff_val}. This means exactly with our terminology that the derivation is a \textbf{Hardy type derivation} (Definition \ref{defi:hardy_deriv}). Hardy fields are key tools for the study of asymptotic problems. For examples, take $\R(x)$ or $\R(x^{\R},\exp(x)^\R)$ or the above cited Hardy's log-exp functions. Also, for any  o-minimal structure over $\R$, the germs at $+\infty$ of unary definable functions form a Hardy field \cite{vdd:aschenbrenner:asymptotic}. As another important example, in the context of an isolated singularity of a real analytic vector field, take the field of meromorphic functions evaluated on a germ of \emph{non oscillating} integral curve \cite{cano_moussu_rolin}. 

Hardy fields may also be generated by adjoining to a given Hardy field certain solutions of differential equations \cite{rosenlicht:hardy_fields,ros:order2-ODE, kuhl:Hensel-lemma, p-crampe-rib:diff-equ-val-fields}. Note that by  \cite{bosher:hardy} there exist also infinite rank Hardy fields that are not \emph{exponentially bounded}, i.e. for which any element is bounded by some iterate of $\exp$. In \cite{vdd:asch:H-fields-liouv-ext,vdd:asch:liouv-cl-H-fields}, the authors have resumed and enhanced the valuation theoretic approach of M. Rosenlicht in the context of \emph{ordered differential fields}: the notion of \textbf{H-fields} is an axiomatized version of Hardy fields, as well as is the notion of \emph{valued field with a Hardy type derivation}. More precisely, the derivation in such H-fields is of Hardy type (see (HF1), (HF2) following Definition \ref{defi:hardy_deriv}). For a nice survey on these topics, see \cite{vdd:aschenbrenner:asymptotic}. \\

How far can one push the connection between Hardy fields and generalized series fields ? For instance, we showed in  \cite{matu:puiseux-diff_rg-fini} that one can solve differential equations over finite rank generalized series similarly as over finite rank Hardy fields. In \cite[Corollary 3.12]{dmm:LE-series}, the authors prove that the infinite rank Hardy field of the above cited o-minimal structure $\R_{an,\exp}$ (the field of reals expanded by the global exponential and restricted analytic functions \cite{dmm:exp}) embeds as ordered differential field into the field of \emph{logarithmic-exponential series}. Recently, J. van der Hoeven introduced in \cite{vdh:transserial-hardy-fields} the notion of a \emph{transserial Hardy field}, namely a Hardy field which is differentially order isomorphic to a subfield of the field of \emph{transseries}. There he proves \cite[Theorem 9]{vdh:transserial-hardy-fields} \cite[Theorem 1.1]{asch-vdd-vdh:mod-theo-transseries} that : \emph{the field of transseries that are differentially algebraic over $\mathbb{R}$ is isomorphic as ordered differential field to some Hardy field.} 
Even more recently \footnote{The author thanks Matthias Aschenbrenner for having indicated him this preprint}, M. Aschenbrenner, L. van den Dries and J. van der Hoeven state the following result for H-fields - so in particular for Hardy fields - \cite[Theorem 4.1]{asch-vdd-vdh:mod-theo-transseries}: \emph{every real closed H-field has an immediate maximally valued H-field extension.} Recall that the derivation in an H-field is a derivation of Hardy type (\ref{defi:hardy_deriv}).  By Kaplansky's embedding Theorem \ref{theo:kap}, this implies that:
\begin{theorem}\label{theo:kap-diff}
Any H-field  is analytically and differentially isomorphic to a subfield of a suitable field of generalized series endowed with \emph{some} Hardy type derivation.
\end{theorem}
The authors assert that the derivation on such differential maximally valued immediate extension needs not be unique (see the proof of their Proposition 5.4), whereas the extension is unique as a valued field by Kaplansky's work. Now follows an important question: what kind of derivation can we expect for the generalized series field? For generalized series, one may not only require the derivation to mimic the valuative properties of the derivation in a Hardy field. One may expect the derivation to  behave the same way as the derivation for the classical formal power series: commutation with infinite sums, real powers go to the coefficients.  This is what we call  a \textbf{series derivation}: see Definition \ref{defi:series-deriv}. Note also that Theorem \ref{theo:kap-diff} does not say anything about immediate maximally valued extensions of Hardy field \emph{within} the class of Hardy fields: it says that any Hardy field admits an immediate maximal H-field extension. In fact, as far as the author knows, there is no example of a maximally valued Hardy field. Together with S. Kuhlmann, we propose the following conjecture, problem and questions:
\begin{description}
 \item[Differential Kaplansky embedding conjecture for Hardy fields.] \emph{Any Hardy field is analytically and differentially isomorphic to a subfield of a suitable generalized series field endowed with a \emph{series derivation of Hardy type}}.
\item[Differential Kaplansky embedding problem.] \emph{Describe which valued differential fields are analytically and differentially isomorphic to a subfield of a suitable generalized series field endowed with a \emph{series derivation of Hardy type}}.
\item[Question 1.] On what subfield of a generalized field is it sufficient to define a series derivation of Hardy type to ensure that it extends uniquely ?  In particular, does the series derivation of Hardy type naturally defined on the field of transseries (grid-based or well-ordered; see Section \ref{sect:EL-series}) extends uniquely to the maximally valued extension ? 
\item[Question 2.] Can we obtain the same result as in Theorem \ref{theo:kap-diff} for valued field endowed with a Hardy type derivation (Definition \ref{defi:hardy_deriv}) ? Or with a differential valuation in the sense of Rosenlicht (only axioms (HD1), (HD2) of \ref{defi:hardy_deriv})  ?
\end{description}

 In the following Sections \ref{sect:deriv} and \ref{sect:hardy-deriv}, the author will survey the results in \cite{matu-kuhlm:hardy-deriv-gener-series} describing how to endow generalized series fields with such \textbf{series derivations of Hardy type}. 
In the case where the generalized series field is ordered, it can thus be endowed with a H-field structure. This has already been done in \cite[Section 11]{vdd:asch:liouv-cl-H-fields} in the particular case of a value group $\Gamma$ divisible with a property called (*) (i.e.  admitting a valuation basis in \cite{kuhl:ord-exp}) and carrying an \emph{asymptotic couple} in the sense of \cite{rosenlicht:val_gr_diff_val2}. In \cite{matu-kuhlm:hardy-deriv-EL-series} we continued our study in the case of \emph{exponential-logarithmic series fields} in the sense of \cite{kuhl:ord-exp} (see Sections \ref{sect:prelog} and \ref{sect:EL-series}), to provide a large family of  \emph{exponential H-fields}. This has already been done for LE-series field and certain transseries fields: see Section \ref{sect:EL-series} and \cite{vdd:LE-pow-series, kuhl:ord-exp, vdh:transs_diff_alg} for details.

\begin{remark}
As will be shown in the next section, such series derivation for generalized series is determined by the definition of the logarithmic derivatives of a certain family of (generalized) monomials. Consider a Hardy field $\mathbb{H}$. Any Kaplansky embedding (\ref{theo:kap}) $f: \mathbb{H}\rightarrow k((\Gamma))$ of $\mathbb{H}$ seen as a valued field will fix these monomials and their logarithmic derivatives. So, the derivation $\delta:=f\circ d$ on $f(\mathbb{H})$ inherited from the derivation $d$ on $\mathbb{H}$ extends \emph{uniquely} to a series derivation on $k((\Gamma))$. So, another way of formulating our conjecture is:
\begin{description}
    \item[Differential Kaplansky embedding conjecture version 2]\emph{For any Hardy field $\mathbb{H}$, there is a Kaplansky embedding (\ref{theo:kap}) $f:\mathbb{H}\hookrightarrow k((\Gamma))$ such that the resulting derivation on $f(\mathbb{H})$ lifts to a Hardy type series derivation on $k((\Gamma))$.}\end{description}
\end{remark}

\section{What kind of derivations for generalized series ?}\label{sect:deriv}

We follow the ideas in \cite{matu-kuhlm:hardy-deriv-gener-series,matu-kuhlm:hardy-deriv-EL-series}, but with \emph{additive notations} as for the classical Krull valuations. 

By a \textbf{derivation} on a field $K$, we mean a map $d:K\rightarrow K$ which is linear and verifies the Leibniz rule : $d(ab)=d(a)b+ad(b)$. The derivation for the usual power series verifies two further key properties :
\begin{itemize}
\item  it \emph{commutes with infinite sums} $d\left(\displaystyle\sum_n a_n x^n\right)=\displaystyle\sum_n a_n d(x^n)$;
\item it \emph{generalizes the Leibniz rule to real powers} $d(x^\alpha)=\alpha x^{\alpha-1}d(x)=\alpha x^\alpha d(x)/x$, with in particular $d(1)=0$.
\end{itemize} 
How can we generalize these properties to the case of our fields $k((\Gamma))$ ?
If we impose that $d$ is \textbf{strongly linear}, that is: $$d\left(\displaystyle\sum_\alpha a_\alpha t^\alpha\right)=\displaystyle\sum_\alpha a_\alpha d(t^\alpha),$$ two questions arise:
\begin{enumerate}
\item how do we define $d(t^\alpha)$, $\alpha\in \Gamma$, where $\Gamma$ is an arbitrary ordered Abelian group ?
\item how do we ensure that $\displaystyle\sum_\alpha a_\alpha d(t^\alpha)$ is itself well-defined and in $k((\Gamma))$ ?
\end{enumerate}
To answer question 1, we will apply the above cited Hahn's embedding theorem (\ref{theo:hahn1}). To any totally ordered (non trivial) Abelian group $(\Gamma,\leq )$, one can associate its \textbf{skeleton} $[\Phi,(A_\phi)_{\phi\in\Phi}]$ as follows. As before $\Phi$ denotes the ordered set of all Archimedian equivalence classes. The map $v_\Gamma:\Gamma\rightarrow\Phi\cup\{\infty\}$ defined by  $v_\Gamma(\alpha):=[\alpha]=\phi$ is the  corresponding \textbf{natural valuation} on $\Gamma$ (as ordered group). To any value $\phi\in\Phi$ corresponds two subgroups of $\Gamma$: $C_\phi:=\{\alpha\in\Gamma\ |\ v_\Gamma(\alpha)\geq \phi\}$ and $D_\phi:=\{\alpha\in\Gamma\ |\ v_\Gamma(\alpha)>\phi\}$, with the following properties:
\begin{itemize}
	\item $D_\phi\subsetneq C_\phi$;
	\item $(0)\subset \cdots D_\phi\subsetneq C_\phi\cdots \subset \Gamma$ forms the chain of \emph{isolated subgroups} of $\Gamma$;
	\item $A_\phi:=C_\phi/D_\phi$ is an Archimedian group, and therefore is (order isomorphic to) some subgroup of $\R$  (Hölder's theorem).
\end{itemize}
The groups $A_\phi$, $\phi\in\Phi$, are usually called the  \textbf{ribs} and $\Phi$ the \textbf{spine} of $\Gamma$. 

 Given a skeleton $[\Phi,(A_\phi)_{\phi\in\Phi}]$, one can always build the corresponding \textbf{Hahn group}, say $\displaystyle\sum_{\phi\in\Phi} A_\phi$. For any $\phi\in\Phi$, consider its index map $\mathds{1}_\phi:\Phi\rightarrow \{0,1\}$, $\phi^{-1}(1)=\{\phi\}$. An element of $\displaystyle\sum_{\phi\in\Phi} A_\phi$ is a formal sum $\displaystyle\sum_{\phi\in\Phi}\alpha_\phi \mathds{1}_\phi$ with $\alpha_\phi\in A_\phi$, $\phi\in\Phi$, and the support $\mathrm{supp}\ \alpha:=\{\phi\ |\ \alpha_\phi\neq 0\}$ that is well-ordered. The group law is the componentwise sum and the total ordering is lexicographical. We can give now a more precise version of Theorem \ref{theo:hahn1} (see e.g.  \cite{fuchs:partial_ord}, \cite[Theorem 4.C]{glass:part-ordered-gr} or \cite[Theorem 0.26]{kuhl:ord-exp}).
\begin{theorem}[Hahn's Embedding]\label{theo:hahn2}
$(\Gamma,\leq)$ embeds as an ordered group in the Hahn group $\displaystyle\sum_{\phi\in\Phi}\overline{A}_\phi$ where $\overline{A}_\phi$ is the divisible closure of $A_\phi$ in $\R$ (i.e. the rational vector subspace of $\R$ generated by $A_\phi$). 
\end{theorem}
Note that the ordering of an ordered group extends uniquely to its divisible closure.  \emph{From now on, we consider $\Gamma$ as a subgroup of a given Hahn group $\displaystyle\sum_{\phi\in\Phi}\overline{A}_\phi$, which is itself a subgroup of $\displaystyle\sum_{\phi\in\Phi}\R$}. So, any element $\alpha\in\Gamma$ is written $\alpha=\displaystyle\sum_\phi \alpha_\phi \mathds{1}_\phi$, and $v_\Gamma(\alpha)=\min(\mathrm{supp}\ \alpha)\in\Phi\cup\{\infty\}$.

Returning to the generalized series field $k((\Gamma))$, we write the generalized monomials as: $$t_\phi:=t^{\mathds{1}_\phi}\ \ \ \ \mathrm{and}\ \ \ \ t^\alpha=t^{\sum_\phi \alpha_\phi \mathds{1}_\phi}:=\prod_\phi  t_\phi^{\alpha_\phi}.$$ The multiplication rule is componentwise, and we will also speak of the support $\textrm{supp}\ t^\alpha$ of such formal product: $\textrm{supp}\ t^\alpha:=\textrm{supp}\ \alpha=\{\phi\ |\ \alpha_\phi\neq 0\}$ which is well-ordered in $\Phi$. Note that, in the case where the sum and the product are finite, one has a true equality $\sum_\phi \alpha_\phi \mathds{1}_\phi = \prod_\phi  t_\phi^{\alpha_\phi}$ between elements of $k((\Gamma))$. 

Now we define the derivative of $t^\alpha$, by imposing that it verifies a \textbf{strong Leibniz rule}, that is: $$d(t^\alpha)=t^\alpha\displaystyle\sum_\phi \alpha_\phi \frac{d(t_\phi)}{t_\phi}\  \textrm{ with in particular }\ d(t^0)=d(1):=0.$$ Subsequently, three new questions arise:
\begin{enumerate}
\item[3.] how do we define $d(t_\phi)$ for $\phi\in\Phi$, with $\Phi$ being an arbitrary ordered set ?
\item[4.] how do we make sense of $\alpha_\phi$ as coefficient in the series ?
\item[5.] how do we ensure that $\displaystyle\sum_\phi \alpha_\phi \frac{d(t_\phi)}{t_\phi}$ is itself well-defined and in $k((\Gamma))$ ?
\end{enumerate}

To answer question 3, one can set $d(t_\phi)\in k((\Gamma))$. In other words, one can pick \emph{a priori} any map $d:\tilde{\Phi} \rightarrow k((\Gamma))\backslash\{0\}$ where $\tilde{\Phi}:=\{t_\phi,\ \phi\in\Phi\}$.

To answer question 4, we need to impose that $\alpha_\phi\in k$ for any $\alpha\in\Gamma$ and $\phi\in \mathrm{Supp}\ \alpha\subset \Phi$. In other words, $k$ has to contain the union of the groups $\overline{A}_\phi$ in $\R$.  For simplicity, one can take $k=\R$ (as we did in \cite{matu-kuhlm:hardy-deriv-gener-series, matu-kuhlm:hardy-deriv-EL-series}) or $k=\C$. 

\begin{example}[The finite rank case]
In the finite rank case, say rank $r$ for some $r\in\N$, the given answers to questions 3 and 4 also solve questions 2 and 5. Indeed, write any element $\alpha\in\Gamma$ as $t^\alpha=t_1^{\alpha_1}t_2^{\alpha_2}\cdots t_r^{\alpha_r}$, and for any $a\in k((\Gamma))$:
$$\begin{array}{lcl}
 d(a)&=&d\left(\displaystyle\sum_{\alpha}a_\alpha t^\alpha\right)\\
&=& \displaystyle\sum_{\alpha}a_\alpha d(t^\alpha)\ \textrm{ (by strong linearity)}\\
&=&\displaystyle\sum_\alpha a_\alpha t^\alpha\left(\alpha_1\displaystyle \frac{d(t_1)}{t_1}+\cdots+\alpha_r\displaystyle\frac{d(t_r)}{t_r }\right)\ \textrm{ (by strong Leibniz rule)}\\
&=&\displaystyle\frac{d(t_1)}{t_1}\sum_\alpha (a_\alpha\alpha_1) t^\alpha+\cdots+\displaystyle\frac{d(t_r)}{t_r}\sum_\alpha (a_\alpha\alpha_r)t^\alpha
\end{array}$$
So $d(a)$ is well-defined whatever value in  $k((\Gamma))$ we choose for the $\displaystyle\frac{d(t_i)}{t_i}$'s. 
\end{example}

It remains to answer questions 2 and 5 in the infinite rank case.
\begin{definition}\label{defi:series-deriv}
We call \textbf{series derivation} on $k((\Gamma))$ any derivation $d$ on $k((\Gamma))$ that is strongly linear (solving question 2)  and that verifies a strong Leibniz rule (solving question 5).
\end{definition}
Thus the problem consists in finding a criterion on the map $d:\tilde{\Phi}\rightarrow k((\Gamma))$ so that the formulas in questions 2 and 5 involve always \textbf{summable families} of series. A family of series is said to be summable if the union of their support is well-ordered and if for any value in this union there are only finitely many series which have that value in their support. The desired criterion is stated and proved in  \cite[Theorem 3.7]{matu-kuhlm:hardy-deriv-gener-series}, using Ramsey's theory type arguments \cite[Exercise 7.5, p. 112]{rosen:lin-ord}: see Theorem \ref{theo:series-deriv}. Since this criterion is technical, firstly we consider for simplicity a particular case for which series derivations can be explicitly built. We suppose also for simplicity that $A_\phi=\R$ for any $\phi\in\Phi$. 

\begin{definition}
A map $\sigma:\Phi\rightarrow \Phi$ is said to be a \textbf{right shift} if it is order preserving and for any $\phi\in\Phi$, $\sigma(\phi)>\phi$.
\end{definition}


\begin{theorem}\label{theo:series-deriv1}
Suppose that $\Phi$ carries a right shift $\sigma:\Phi\rightarrow \Phi$. Define 
$d:\tilde{\Phi}\rightarrow k((\Gamma))$ by:
$$\frac{d(t_\phi)}{t_\phi}:= t^{\theta_\phi}=\displaystyle\prod_{n\geq 1}t_{\sigma^n(\phi)}^{\theta_{\phi,n}}$$
where $\theta_{\phi,n}\in\R$ with in particular $\theta_{\phi,1}<0$.\\
Then $d$ extends to a series derivation on $k((\Gamma))$ via strong Leibniz rule and strong linearity.
\end{theorem}
\begin{proof}
Consider a monomial $ t^\alpha=\displaystyle\prod_{\phi\in\Phi}t_\phi^{\alpha_\phi}$ for some $\alpha\in\Gamma$ and apply the strong Leibniz rule to compute its derivative:
\begin{center}
$\begin{array}{lcl}
 d(t^\alpha)&=&t^\alpha\displaystyle\sum_{\phi\in\Phi}\alpha_\phi \displaystyle\frac{d(t_\phi)}{t_\phi}\\
&=&t^\alpha\displaystyle\sum_{\phi\in\Phi}\alpha_\phi \displaystyle\prod_{n\geq 1}t_{\sigma^n(\phi)}^{\theta_{\phi,n}}
\end{array}$
\end{center}
Note that, for any $\phi<\psi$, $\theta_\phi=v\left(\displaystyle\prod_{n\geq 1}t_{\sigma^n(\phi)}^{\theta_{\phi,n}}\right) <v\left(\displaystyle\prod_{n\geq 1}t_{\sigma^n(\psi)}^{\theta_{\psi,n}}\right)=\theta_\psi$ since $\sigma(\phi)<\sigma(\psi)$ and $\theta_{\phi,1}<0$. Therefore, the support of the sum in the right hand expression is well-ordered as well as the support of $\alpha$: we obtain a well defined series in $k((\Gamma))$.\\
Consider a series $a=\displaystyle\sum_{\alpha\in\Gamma}a_\alpha t^\alpha$ and apply the strong linearity to compute its derivative:
$$\begin{array}{lcl}
 d(a)&=&\displaystyle\sum_{\alpha\in\Gamma}a_\alpha d(t^\alpha)\\
&=&\displaystyle\sum_{\alpha\in\Gamma}a_\alpha t^\alpha\displaystyle\sum_{\phi\in\Phi}\alpha_\phi \displaystyle\frac{d(t_\phi)}{t_\phi}\\
&=&\displaystyle\sum_{\alpha\in\Gamma}\displaystyle\sum_{\phi\in\Phi}a_\alpha\alpha_\phi  t^{\alpha+\theta_\phi}
\end{array}$$
Suppose that the family $(a_\alpha\alpha_\phi  t^{\alpha+\theta_\phi})$ in the right hand sum is not summable, say there is a decreasing sequence $(\alpha_n+\theta_{\phi_n})_{n\in\N}$. Since $\mathrm{Supp}\ a$ is well-ordered, the sequence $(\alpha_n)_{n\in\N}$ is increasing, non ultimately stationary. So we would have for infinitely many $n$: 
	$$0<\alpha_{n+1}-\alpha_n\leq \theta_{\phi_n}-\theta_{\phi_{n+1}}$$
The corresponding sequence $(\phi_n)_{n\in\N}$ is strictly decreasing. So there exists $n_2\in\N$ such that $\phi_{n_2}\notin\mathrm{supp}\ \alpha_1$. So $\phi_{n_2}\in\mathrm{supp}\ (\alpha_{n_2}-\alpha_1)$, which implies that $v_\Gamma(\alpha_{n_2}-\alpha_1)\leq \phi_{n_2}$. But, by construction, we have $v_\Gamma(\theta_{\phi_1}-\theta_{\phi_{n_2}})=\sigma(\phi_{n_2})>\phi_{n_2}$. This is in contradiction with the fact that 	$0<\alpha_{n_2}-\alpha_1\leq \theta_{\phi_1}-\theta_{\phi_{n_2}}$.
\end{proof}

Note that we could have built other derivations by choosing a more complicated formula for $d(t_\phi)/t_\phi$, for instance a two monomials expression: $$\frac{d(t_\phi)}{t_\phi}:= t^{\theta_\phi}+t^{\tau_\phi} =\displaystyle\prod_{n\geq 1}t_{\sigma^n(\phi)}^{\theta_{\phi,n}}+ \displaystyle\prod_{n\geq 1}t_{\sigma^n(\phi)}^{\tau_{\phi,n}}$$ with $\theta_\phi<\tau_\phi<0$. In this case, for any $\phi_1<\phi_2$ we have  $\theta_1<\theta_2$ and $\tau_1<\tau_2$. We also have $v_\Gamma(\theta_1-\theta_2)>\phi_1$, $v_\Gamma(\tau_1-\tau_2)>\phi_1$, $v_\Gamma(\theta_1-\tau_2)>\phi_1$ and $v_\Gamma(\tau_1-\theta_2)>\phi_1$.  In fact, in the general case when $d(t_\phi)\in k((\Gamma))$, the same ingredients as in the preceding  proof apply provided the fact that, roughly speaking, \emph{for almost all} $\phi_1<\phi_2$:
\begin{itemize}
\item the right shift $\sigma$ lifts to a right shift between corresponding elements of $\mathrm{Supp}\ \displaystyle\frac{d(t_{\phi_1})}{t_{\phi_1}}$ and $\mathrm{Supp}\ \displaystyle\frac{d(t_{\phi_2})}{t_{\phi_2}}$. 
\begin{center}
\includegraphics[scale=0.6]{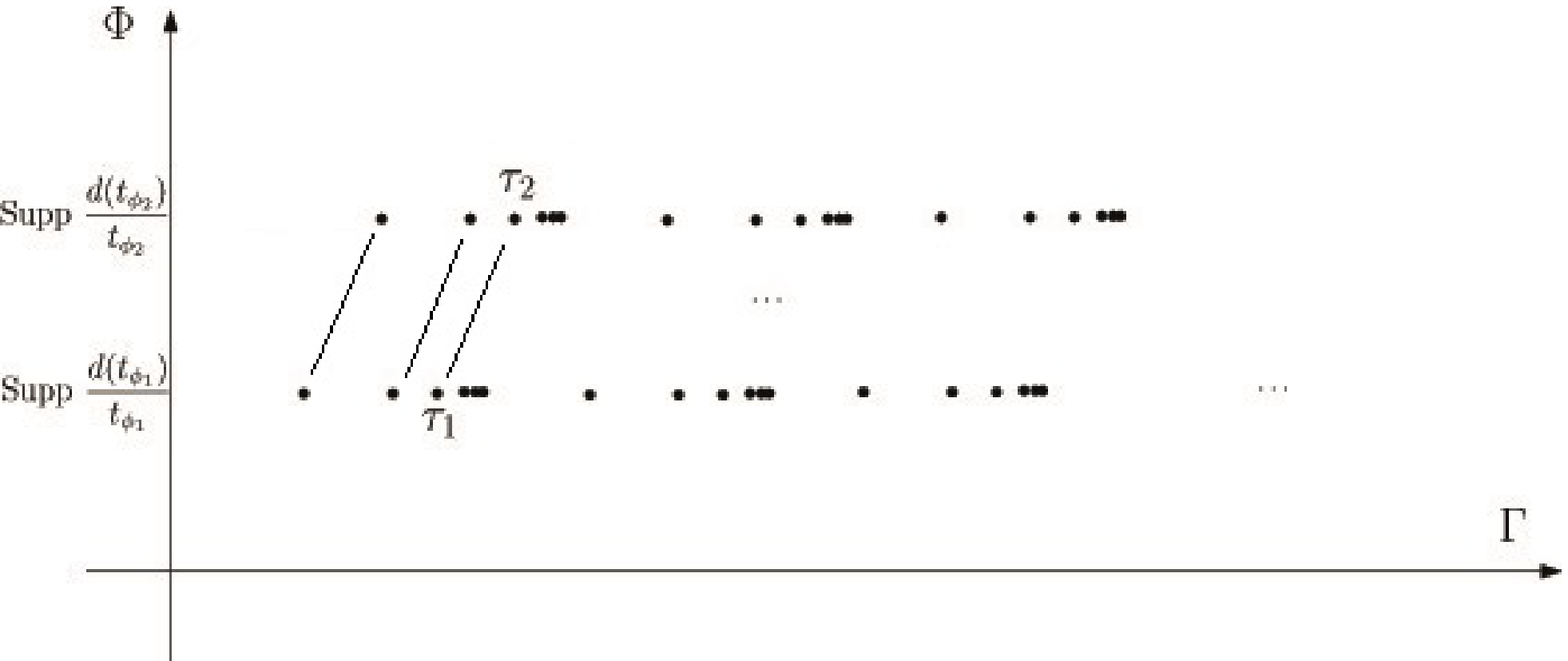}
\end{center}
\item \emph{for almost all} $\tau_1\in\mathrm{Supp}\ \displaystyle\frac{d(t_{\phi_1})}{t_{\phi_1}}$ and $\tau_2\in \mathrm{Supp}\ \displaystyle\frac{d(t_{\phi_2})}{t_{\phi_2}}$, we have $v_\Gamma(\tau_1-\tau_2)>\phi_1$.
\end{itemize}
The precise criterion is as follows \cite[Theorem 3.7]{matu-kuhlm:hardy-deriv-gener-series}:

\begin{theorem}\label{theo:series-deriv}
 A map $d :
\tilde{\Phi}\rightarrow k((\Gamma))\backslash\{0\}$ extends to a series
derivation on $k((\Gamma))$ via strong Leibniz rule and strong linearity if and only if both of the following conditions \emph{hold}:
\begin{description}
\item[{\bf(C1)}] for any strictly increasing sequence $(\phi_n)_{n\in\mathbb{N}}\subset\Phi$ and any sequence
 $(\tau_n)_{n\in\mathbb{N}}\subset\Gamma$ with for any $n$, $\tau_n\in\mathrm{Supp}\ \displaystyle\frac{d(t_{\phi_n})}{t_{\phi_n}}$, then $(\tau_n)_{n\in\mathbb{N}}$ cannot be decreasing (i.e. there is $N\in\N$ such that $\tau_{N+1}>\tau_N$).
\item[{\bf(C2)}] For any strictly decreasing sequences $(\phi_n)_{n\in\mathbb{N}}\subset\Phi$ and
$(\tau_n)_{n\in\mathbb{N}}\subset\Gamma$ such that for any $n$,
$\tau_n\in\mathrm{Supp}\ \displaystyle\frac{d(t_{\phi_n})}{t_{\phi_n}}$, there is $N\in\N$ such that $v_\Gamma(\tau_{N+1}-\tau_N) >\phi_{N+1}$.
\end{description}
\end{theorem}

\begin{example}\label{ex:hardy-series}
We illustrate (\ref{theo:series-deriv}) by the following basic example. Consider the Hardy field of germs at $+\infty$ of real functions $\mathbb{H}=\R(\log(x)^{\R},x^{\R},\exp(x)^{\R})$, which can be viewed also as a subfield of a field $\R((\Gamma))$ of generalized series of rank 3, with e.g. $t_1=\exp(-x)$, $t_2=1/x$, $t_3=1/\log(x)$ and $\Gamma=\R\times_{\mathrm{lex}}\R\times_{\mathrm{lex}}\R$. As an illustration of this embedding, take for instance the germ of $1/(\exp(x)-x)=\exp(-x)/(1-\exp(-x)x)$. It goes to $t_1/(1-t_1t_2^{-1})= t_1\sum_{n\in\N}t_1^{n}t_2^{-n}=\sum_{n\in\N}t_1^{n+1}t_2^{-n}$ since one has that $v(t_1t_2^{-1})>(0,0,0)$, or equivalently $\lim_{x\rightarrow +\infty} \exp(-x)x=0$.
Note that, for any $n\in\N_{>0}$, $0<t_1^n<t_2^n<t_3^n<\R$ for the ordering in $\mathbb{H}$ corresponding to $(0,0,0)<nv(t_3)=(0,0,n)<nv(t_2)=(0,n,0)<nv(t_1)=(n,0,0)<\infty$ in $\Gamma$ and  $\phi_1=v_\Gamma(v(t_1))=[(1,0,0)]<\phi_2=v_\Gamma(v(t_2))=[(0,1,0)]< \phi_3=v_\Gamma(v(t_3))=[(0,0,1)]<\infty$ for the Archimedian equivalence classes of $\Gamma$. We compute in $\mathbb{H}$:
\begin{center}
$\left\{\begin{array}{lcccll}
d(t_1)&=&-\exp(-x)&=& -t_1&\mathrm{with\ support}\ \{(1,0,0)\} \\
d(t_2)&=&\displaystyle\frac{-1}{x^2}&=&-t_2^2 &\mathrm{with\ support}\ \{(0,2,0)\}\\
d(t_3)&=&\displaystyle\frac{-1}{x(\log(x))^2}&=& -t_2t_3^2 & \mathrm{with\ support}\ \{(0,1,2)\}\\
\end{array}\right.$\\

$\left\{\begin{array}{lcccll}
\displaystyle\frac{d(t_1)}{t_1}&=&-1&& &\mathrm{with\ support}\ \{(0,0,0)\} \\
\displaystyle\frac{d(t_2)}{t_2}&=&\displaystyle\frac{-1}{x}&=&-t_2 &\mathrm{with\ support}\ \{(0,1,0)\}\\
\displaystyle\frac{d(t_3)}{t_3}&=&\displaystyle\frac{-1}{x\log(x)}&=& -t_2t_3 & \mathrm{with\ support}\ \{(0,1,1)\}\\
\end{array}\right.$
\end{center}
Consider for instance $t_2$ and $t_3$. One has that $v\left(\displaystyle\frac{d(t_2)}{t_2}\right)=(0,1,0) <(0,1,1)=v\left(\displaystyle\frac{d(t_3)}{t_3}\right)$ and  their difference equals $(0,0,1)$: $v_\Gamma(0,0,1)=\phi_3$ which is indeed bigger than $\phi_2$ in $\Phi$. 
\end{example} 

Our criterion in (\ref{theo:series-deriv}) permits to build families of series derivations for a large class of generalized series fields \cite[Section 5]{matu-kuhlm:hardy-deriv-gener-series}. Moreover we will be able to obtain in certain cases series derivations of Hardy type, as will be shown in (\ref{ex:generaux}). One of the central notions in such explicit constructions of derivations is the one of \emph{right-shift} as is already illustrated in 
Theorem \ref{theo:series-deriv1}. This notion was already central in the construction of pre-logarithms on generalized series fields: see Section \ref{sect:prelog} and \cite{kuhl:ord-exp} for details.

\section{On Hardy type series derivations}\label{sect:hardy-deriv}
\begin{definition}\label{defi:hardy_deriv} Let $(K,v,d)$ be a valued differential field. Denote by $\mathcal{O}_v$ the valuation ring and $\mathfrak{m}_v$ its maximal ideal. The derivation $d\ :\ K\rightarrow K\ $ is said to be a {\bf Hardy type derivation} if :
\begin{description}
\item[(HD1)] $\mathcal{O}_v=\ker d + \mathfrak{m}_v$;
\item[(HD2)] $d$ verifies \textbf{l'Hospital's rule}:
$\forall a,b\in K\backslash\{0\}$ with $v(a)\neq 0$ and $v(b)\neq 0$, \begin{center}
$v(a)\leq v(b)\Leftrightarrow v(d(a))\leq v(d(b))$.
\end{center}
\item[(HD3)] the \textbf{logarithmic derivation is compatible with the valuation}: \emph{$\forall  a,b\in K$ \begin{center}
$|v(a)|>|v(b)|>0\Rightarrow v\left(\displaystyle\frac{d(a)}{a}\right)\leq v\left( \displaystyle\frac{d(b)}{b}\right)$\\
with: $\ \ v\left(\displaystyle\frac{d(a)}{a}\right)=v\left(\displaystyle\frac{d(b)}{b} \right)\Leftrightarrow v(a)\sim_+ v(b)$
\end{center}}
\end{description}
\end{definition}

Axioms (HD1) and (HD2) are those which define a \emph{differential
valuation} in the sense of M. Rosenlicht (\cite[Definition p. 303]{rosenlicht:diff_val}). Axiom (HD3) corresponds to the Principle (*) in \cite[p. 992]{rosenlicht:val_gr_diff_val2}. This principle is itself an abstract version of properties obtained in \cite[Propositions 3 and 4]{rosenlicht:rank} and \cite[Principle (*) p. 314]{rosenlicht:diff_val} in the context of Hardy fields. Note that these axioms hold for fields that are not necessarily ordered, whereas H-fields (see below) and Hardy fields are. As M. Rosenlicht does, we believe that the tools we develop here may be used in the more general context of valued differential fields: see the various examples in \cite{rosenlicht:diff_val}, in particular Example 2 and 10.

The question we want to address here is: \begin{description}
\item \emph{How can we endow $k((\Gamma))$ with a Hardy type series derivation ?} \end{description} As we showed in \cite[Theorem 4.3 and Corollary 4.4]{matu-kuhlm:hardy-deriv-gener-series}, the answer depends only on the values $v(d(t_\phi)/t_\phi$), say $\theta_\phi$, $\phi\in\Phi$:
\begin{theorem}\label{theo:hardy-deriv} A series derivation $d$ on $k((\Gamma))$ verifies \emph{(HD2)} and \emph{(HD3)} if and only if
the following condition holds:\begin{description}
\item[(H3)] $\forall\phi_1,\phi_2\in\Phi$,\  $\phi_1<\phi_2\ \Rightarrow\ 
\theta_{\phi_1}<\theta_{\phi_2}\ and\ v_\Gamma\left(\theta_{\phi_1}-\theta_{\phi_2}\right)>\phi_1$.
\end{description}
In this case (HD1) holds with $k=\ker d$.
\end{theorem}
In other words, for the given hypothesis, the corresponding map $\phi\mapsto \theta_\phi$ is order-preserving and the maps $\theta_{\phi_1}\mapsto\theta_{\phi_2}$ and $\phi_1\mapsto v_\Gamma\left(\theta_{\phi_1}-\theta_{\phi_2}\right)$ are \emph{right-shifts}. As an illustration, consider the example in (\ref{ex:hardy-series}).\\

Concerning ordered differential fields, in \cite{vdd:asch:H-fields-liouv-ext} is developed the notion of \emph{H-field} which generalizes the one of Hardy field. An \textbf{H-field} is an ordered field (denote by $v$ the natural convex valuation on it) endowed with a derivation $d:K\rightarrow K$ such that:
\begin{description}
    \item[(HF1)] $\mathcal{O}_v=\ker d + \mathfrak{m}_v$;
\item[(HF2)] for any $f\in K_{>0}$, $v(f)<0 \Rightarrow d(f)>0$.
\end{description}
Note that (HD1) is (HF1). Properties (HD2) and (HD3) also hold for H-fields \cite[Lemma 1.1 and 2.2]{vdd:asch:H-fields-liouv-ext}. \\
In the case where $k$ is ordered, one can order $k((\Gamma))$ lexicographically, so that the valuation ring $k[[\Gamma_{\geq 0}]]$ is the convex hull of $k$. Suppose that $k((\Gamma))$ is endowed with a series derivation $d$ of Hardy type.  In this context:
\begin{description}
    \item[]\emph{ $k((\Gamma))$ is an H-field if and only if for any $\phi\in\Phi$, $\displaystyle\frac{d(t_\phi)}{t_\phi} <0$, that is, the leading coefficient of $\displaystyle\frac{d(t_\phi)}{t_\phi}$ is negative. }
\end{description}
\noindent Indeed, for any series $a>0$ with $v(a)<0$, denote $a=a_\alpha t^\alpha+\cdots$ where $v(a)=\alpha$ and $a_\alpha>0$. Denote $v_\Gamma(\alpha)=\phi$ and the coefficient of $\mathds{1}_\phi$ in $\alpha$ by $\alpha_0$. Note that $\alpha_0<0$ since $v(a)=\alpha<0$. By the strong linearity, the strong Leibniz rule and (H3), we have: 
\begin{equation}\label{equ:deriv-log}
d(a)=a_\alpha d(t^\alpha)+d(\cdots)=a_\alpha t^\alpha \left(\alpha_0 \displaystyle\frac{d(t_\phi)}{t_\phi}+\cdots\right)+d(\cdots) 
\end{equation}
So $d(a)$ has same sign as $-\displaystyle\frac{d(t_\phi)}{t_\phi}$. Thus, $d(a)>0$ as required in (HF1) iff $\displaystyle\frac{d(t_\phi)}{t_\phi} <0$.

\begin{example}\label{ex:generaux} We use theorems \ref{theo:series-deriv} and \ref{theo:hardy-deriv} to build general examples of series derivations of Hardy type on $k((\Gamma))$. If we restrict our attention to the value $\theta_\phi=v(d(t_\phi)/t_\phi)$, then (H3) tells us that the map $\phi\mapsto \theta_\phi$ has to be a \emph{section} of $\Phi$ in $\Gamma$.  It says also that for any $\phi_1<\phi_2$, the principal parts of $\theta_{\phi_1}$ and $\theta_{\phi_2}$ up to the component $\mathds{1}_{\phi_1}$ have to be identical. This can be achieved easily in the following cases. We leave the verifications to the reader.
\begin{itemize}
\item If $\Phi$ admits a \emph{right-shift} $\sigma:\Phi\rightarrow\Phi$, then for any $\phi$  pick $\theta_\phi\in\Gamma_{<0}$ such that $v_\Gamma(\theta_\phi)=\sigma(\phi)$ (see e.g. Theorem \ref{theo:series-deriv1}) and set for example $\displaystyle\frac{d(t_\phi)}{t_\phi}=c_{\phi}t^{\theta_\phi}$ for arbitrary $c_{\phi}\in k$. This includes for example the cases where $\Phi$ carries the structure of an ordered group, or where $\Phi$ is a limit ordinal. 
\item If $\Phi$ has a biggest element $\phi_0$ and carries a right-shift $\sigma:\Phi\setminus\{\phi_0\}\rightarrow\Phi$ then we can set the $\theta_\phi$'s and the logarithmic derivatives as before, just completing the definition for $\phi_0$ by setting  $\theta_{\phi_0}:=0$. This includes for example the cases where $\Phi$ is a successor ordinal or where $\Phi$ is isomorphic to an interval of $\R$ with greatest element. 

\end{itemize}
\end{example}

As a prolongation of \cite[Section 5.1]{matu-kuhlm:hardy-deriv-gener-series}, in \cite[Example 6, case 2]{matu-kuhlm-shkop:schanuel-note} we treat completely the case of subsets of $\R$.  Note that this includes any \emph{countable} well-ordered, reverse well-ordered or more generally scattered set (they all embed into $\Q$). It includes also the case of some tricky \emph{dense} subsets of $\R$ as the Dushnik-Miller example (a dense subset of $\R$ that admits no non trivial self-embedding). We invite the reader to read through \cite{rosen:lin-ord} for the various references and classical results. For simplicity we suppose for the rest of this section that $A_\phi=\R$ for any $\phi\in\Phi$.
\begin{proposition}\label{propo:subset-R}
Let $\Phi$ be isomorphic to a subset $S$ of $\R$. 
Then $k((\Gamma))$ carries a series derivation of Hardy type such that for any $\phi\in\Phi$, $\displaystyle\frac{d(t_\phi)}{t_\phi}$ is a monomial.
\end{proposition}
\begin{proof} 
Suppose that $\Phi$ has a greatest element $\phi_0$ and consider an isomorphism $s:\Phi\rightarrow S\subset A_{\phi_0}=\R$. For any $\phi\in \Phi$ pick some element $\theta_\phi=s(\phi) \mathds{1}_{\phi_0}$ of $\Gamma$ and set for example $\displaystyle\frac{d(t_\phi)}{t_\phi}=c_{\phi}t^{\theta_\phi}= c_{\phi}t_{\phi_0}^{s(\phi)}$ where $c_{\phi}\in k$. Note that $v_\Gamma(\theta_{\phi_0})=\phi_0$.\\
Suppose that $\Phi$ has no greatest element. Then choose a sequence $(\phi_n)_{n\in\N_{>0}}$ of elements of $\Phi$ cofinal in $\Phi$, and set $\phi_0:=\inf \Phi$. The intervals $(\phi_n,\phi_{n+1}]\subset \Phi$ form a partition of $\Phi$ (i.e. they cover $\Phi$ without overlapping). Fix a corresponding family of isomorphisms $s_n:(\phi_{n-1},\phi_{n}]\rightarrow S_n\subset A_{\phi_{n+1}}=\R$.  For any $\phi\in (\phi_{n-1},\phi_{n}]$ pick some element $\theta_\phi=s_n(\phi) \mathds{1}_{\phi_{n+1}}+\cdots$ of $\Gamma$ and set for example $\displaystyle\frac{d(t_\phi)}{t_\phi}=c_{\phi}t^{\theta_\phi} = c_{\phi} t_{\phi_{n+1}}^{s_{n}(\phi)}\cdots $ where $c_{\phi}\in k$.
\end{proof}

\begin{remark}
In the second part of the preceding proof, we could have chosen a family of isomorphisms $s_n:(\phi_{n-1},\phi_{n}]\rightarrow S_n\subset A_{\phi_{n}}=\R$. But in this case, to comply the second part of (H3), we must adjoin a principal part to the $\theta_\phi$'s. Pick $\alpha_\phi=s_n(\phi) \mathds{1}_{\phi_{n}}+\cdots$ of $\Gamma$, then set $\theta_\phi:=s_1(\phi_1)\mathds{1}_{\phi_1}+\cdots+ s_{n-1}(\phi_{n-1})\mathds{1}_{\phi_{n-1}}+\alpha_\phi$  and for example $\displaystyle\frac{d(t_\phi)}{t_\phi}=c_{\phi}t^{\theta_\phi} = c_{\phi}t_{\phi_1}^{s_1(\phi_1)} \cdots t_{\phi_{n-1}}^{s_{n-1}(\phi_{n-1})} t_{\phi_{n}}^{s_{n}(\phi_{n})}\cdots $ where $c_{\phi}\in k$. As an illustration, in Example (\ref{ex:hardy-series}), one has $\displaystyle\frac{d(t_2)}{t_2}=\displaystyle\frac{1}{x}=-t_2$ and $\displaystyle\frac{d(t_3)}{t_3}=\displaystyle\frac{-1}{x\log(x)}= -t_2t_3$.
\end{remark}

We can generalize the construction in the preceding proof.

\begin{corollary}
Let $\Phi$ be isomorphic to the concatenation of a family of subsets of $\R$ over any ordered set which admits a right-shift. Suppose that $A_\phi\simeq \R$ for any $\phi\in\Phi$. Then  $k((\Gamma))$ carries a series derivation of Hardy type.
\end{corollary}
Indeed, denote $\Phi=\bigsqcup_{e\in E}S_e$ where $E$ is  the given set with a right-shift $f:E\rightarrow E$. In the proof of (\ref{propo:subset-R}), replace $\N$ by $E$, pick a family of elements $\phi_e\in S_e$ and apply the construction where $S_e$ stands for $S_n$, $\phi_{f(\phi_e)}$ for $\phi_{n+1}$, etc. Note that such set $\Phi$ may not carry itself a right shift. E.g. take $E=\Z$ and, for any $n\neq 0$, $S_n=\{0\}$ and $S_0=\Z_{\leq 0}$: the resulting $\Phi$ is order isomorphic to $\Z_{\leq 0}\bigsqcup\Z$ which does not admit any right shift.\\
%

Some of these abstract examples may be illustrated by or enhanced using germs in some Hardy fields in the spirit of (\ref{ex:hardy-series}). For instance, take the following Hardy fields:
\begin{itemize}
\item $\R(\exp(x)^\R,\exp_2(x)^\R,\ldots,\exp_n(x)^\R,\ldots)$
\item $\R(x^\R,\exp(x^\alpha)^\R;\alpha\in\R_{>0})$
\item $\R(x^\R,\log(x)^\R,\ldots,\log_n(x)^\R,\ldots)$
\item $\R(\ldots,\exp_n(x)^\R,\ldots, \exp_2(x)^\R, \exp(x)^\R,x^\R,\log(x)^\R,\ldots,\log_n(x)^\R,\ldots)$
\end{itemize}
We let the reader verify what kind of Hardy type series derivation may model these cases.

The preceding examples give partial solution to the following natural problem:\begin{description}
\item \emph{ Describe the generalized series fields which can carry a Hardy type series derivation.}
\end{description}
 We believe that a complete answer can be derived from (\ref{theo:hardy-deriv}). This would also help to characterize which groups may belong to an \emph{asymptotic couple} in the sense of \cite{rosenlicht:val_gr_diff_val2,vdd:asch:liouv-cl-H-fields}.

\section{On asymptotic integration, integration and logarithms.}\label{sect:prelog}

\begin{definition}\label{defi:integ}
Let $(K,v,d)$ be a valued differential field, and $a\in K$.
The element $a$ is said to admit an \textbf{asymptotic integral} $b$ if there exists $b\in K\setminus\{0\}$ such that $v(d(b)-a)> v(a)$. The element $a$ is said to admit an \textbf{integral} $b$ if there exists $b\in K\setminus\{0\}$ such that $d(b)=a$.
\end{definition}
For a valued differential field, the existence of a valuation permits to deal with approximate solutions of equations, for instance the basic differential equation corresponding to integration: \begin{equation}\label{equ:integration}
 d(y)=a
\end{equation}
In the context of Hardy fields, the problem of computing asymptotic integrals has been solved by M. Rosenlicht \cite[Proposition 2 and Theorem 1]{rosenlicht:rank}. In \cite{matu-kuhlm:hardy-deriv-gener-series} we observed that his proof applies to the more general context of valued fields with Hardy type derivations:
\begin{theorem}[Rosenlicht]\label{theo:integ-asymp} Let $(K,v,d)$ be a valued differential field with $d$ of Hardy type. Let $a\in K\backslash\{0\}$, then $a$ admits an asymptotic
integral if and only if \begin{center}
 $v(a)\neq  l.u.b.\left\{v\left(\displaystyle\frac{d(b)}{b}\right);\ b\in K\backslash\{0\},\ v(b)\neq 0\right\}$
\end{center}
Moreover, for any such $a$, there exists $u_0\in
K\backslash\{0\}$ with $v(u_0)\neq 0$ such that for any $u\in
K\backslash\{0\}$, $|v(u_0)|> |v(u)|>0$,
then $a.\displaystyle\frac{a u/d(u)}{d(a u/d(u))}$ is an asymptotic integral of $a$.
\end{theorem}
Note that in our context of a generalized series field $k((\Gamma))$ endowed with a series derivation of Hardy type, by the computation (\ref{equ:deriv-log}) in Section \ref{sect:hardy-deriv}, one has  for any $b\in K\backslash\{0\}$ with $v(b)\neq 0$ that:
\begin{center}
$v\left(\displaystyle\frac{d(b)}{b}\right)= v\left(\displaystyle\frac{d(t^\beta)}{t^\beta}\right)= v\left(\displaystyle\frac{d(t_\phi)}{t_\phi}\right)=\theta_\phi$
\end{center}  
where $\beta=v(b)$ and $\phi=v_\Gamma(\beta)$. Therefore in the hypothesis of (\ref{theo:integ-asymp}) we can replace the set  $\left\{v\left(\displaystyle\frac{d(b)}{b}\right);\ b\in k((\Gamma))\setminus\{0\},\ v(b)\neq 0\right\}$ by the set $\left\{\theta_\phi;\ \phi\in\Phi\right\}$.

We can derive from (\ref{theo:integ-asymp}) in our context explicit formulas for the computation of a specific asymptotic integral: the unique one which is a monomial $c_\gamma t^\gamma$  \cite[Corollary 6.3]{matu-kuhlm:hardy-deriv-gener-series} or \cite[Section 4.1]{matu-kuhlm:hardy-deriv-EL-series}:
\begin{corollary}
Let $\alpha\in\Gamma$ with $\alpha\neq \tilde{\theta}$, $\alpha=\alpha_0\mathds{1}_{\phi_0}+\cdots$. There exists
a uniquely determined $\psi_\alpha\in\Phi$ which satisfies
$\alpha - \theta_{\psi_\alpha}= \gamma_0\mathds{1}_{\psi_\alpha}+\cdots$ for some $\gamma_0\neq 0$.\\
 Set $\displaystyle\frac{d(t_{\psi_\alpha})}{t_{\psi_\alpha}}=c_{\psi_\alpha}
t^{\theta_{\psi_\alpha}}+\cdots$. 
Consequently, any series $a=a_\alpha t^\alpha+\cdots$  admits as \textbf{monomial asymptotic integral}:
\begin{center}
$\mathrm{a.i.}(a)= \displaystyle\frac{a_\alpha}{\gamma_0 
c_{\psi_\alpha}}t^{\alpha - \theta_{\psi_\alpha}}$
\end{center}
\end{corollary}

A natural question now is: \begin{description}
\item \emph{What is the relation between asymptotic integration and integration ?} \end{description} The answer is given in \cite[Theorem 55]{kuhl:Hensel-lemma} for spherically complete valued differential fields: \emph{asymptotic integration implies integration}. Applying this result to our context we obtain that:
\begin{corollary}\label{coro:integ}
Let $k((\Gamma))$ be endowed with a series derivation of
Hardy type  $d$. Set
$\tilde{\theta}=l.u.b.\left\{\theta_{\phi}\ ;\
\phi\in\Phi\right\}$. Then any series
$a\in k((\Gamma))$ with $v(a)>\tilde{\theta}$ admits an integral in
$k((\Gamma))$. Moreover $k((\Gamma))$ is closed under integration if
and only if $\tilde{\theta}\notin\Gamma$.
\end{corollary}

A particular case of the problem of integration that we want to solve is the one for the \emph{logarithm}: \begin{equation}\label{equ:log}
d(y)=\displaystyle\frac{d(a)}{a}\Leftrightarrow y=\log|a|+c
\end{equation}
This is the subject of the section 4.2 in \cite{matu-kuhlm:hardy-deriv-EL-series}. There, such solutions $\log$ to (\ref{equ:log}) are called \textbf{pre-logarithms} since they might not be surjective as a real logarithm is. Of course, our study is rooted in the studies of logarithmic and exponential maps in the non-Archimedian context, e.g. \cite{alling:exp_closed_fields, kuhl:ord-exp, ress, vdd:LE-pow-series, vdh:transs_diff_alg}. 

By \cite{alling:exp_closed_fields},  in order to define a pre-logarithm on $k((\Gamma))$, we suppose \emph{from now on} that the coefficient field $k$  carries a logarithm $\log$, e.g. $k=\R$. So in particular $k$ is \emph{ordered}, and therefore so is $k((\Gamma))$ lexicographically. Moreover, the \textbf{logarithm on 1-units} is naturally defined as the following isomorphism of ordered groups:
\begin{equation}\label{eq:log-1-unit}
\begin{array}{lccl}
\log\ :\ &1+k((\Gamma_{>0}))& \rightarrow&k((\Gamma_{>0}))\\
&1+\epsilon&\mapsto&\displaystyle\sum_{n\geq 1}(-1)^{n-1}\displaystyle\frac{\epsilon^n}{n}
\end{array}
\end{equation}

Since any generalized series may be written $a=a_\alpha t^\alpha(1+\epsilon)$, it remains only to define a pre-logarithm on the monic monomials $t^\alpha$. In particular, we are interested in pre-logarithms verifying the \textbf{growth axiom scheme} as a real logarithm does:
\begin{description}
\item[(GA)] $\forall \alpha\in\Gamma_{<0}$, $\ v(\log(t^\alpha))> \alpha$.
\end{description}

In the non differential case, this problem is discussed extensively in \cite{kuhl:ord-exp}. The perspective we adopt in the differential case is to consider pre-logarithms that verify equation (\ref{equ:log}). Moreover, in the context of generalized series fields, we consider pre-logarithms that are \textbf{series morphisms}, i.e. such that:
\begin{description}
\item[(L)] $\forall\alpha=\displaystyle\sum_{\phi\in\Phi}\alpha_\phi\mathds{1}_\phi \in\Gamma,
\ \ \log(t^\alpha)=\displaystyle\sum_{\phi\in\Phi}\alpha_\phi \log(t_\phi)$.
\end{description}
We prove in  \cite[Theorem 4.10]{matu-kuhlm:hardy-deriv-EL-series} that:

\begin{theorem}\label{theo:prelog}
Let $k((\Gamma))$ be endowed with a series derivation of
Hardy type  $d$. Set $\tilde{\theta}=l.u.b.\left\{\theta_{\phi}\ ;\
\phi\in\Phi\right\}$. There exists a unique pre-logarithm  $\log$ on
$k((\Gamma))$ which is a series morphism if and only if the following two conditions hold:\begin{enumerate}
    \item $\tilde{\theta}\notin\displaystyle\bigcup_{\phi\in\Phi}\textrm{Supp}\ \displaystyle\frac{d(t_\phi)}{t_\phi}$;
\item $\forall\phi\in\Phi$, $\forall\tau_{\phi}\in\textrm{Supp}\ \displaystyle\frac{d(t_\phi)}{t_\phi}$,  $v\left(\mathrm{a.i.}(\tau_{\phi})\right)<0$.
\end{enumerate}
 Moreover, this pre-logarithm verifies (GA).
\end{theorem}
In \cite[Corollary 4.13]{matu-kuhlm:hardy-deriv-EL-series}, we obtain explicit formulas for this pre-logarithm.

Generally speaking, the principle is the same as before: to get the good property on $k((\Gamma))$ - here that any $\displaystyle\frac{d(a)}{a}$ has an integral - it suffices to have it for the $\displaystyle\frac{d(t_\phi)}{t_\phi}$'s. Once again, the key property that we use in the proof is the spherical completeness of $k((\Gamma))$, more precisely the existence of a \emph{fixed point principle} as in \cite{kuhl:Hensel-lemma, p-crampe-rib_FP-gen-series}. 

 As a conclusion to this section, note that one can derive from the preceding results methods of construction of pre-logarithms and classes of examples in the same spirit as (\ref{ex:generaux}). Conversely, one may also try to deduce from a given pre-logarithmic structure a corresponding series derivation of Hardy type: see \cite[Section 5]{matu-kuhlm:hardy-deriv-EL-series}
\begin{example}\label{ex:pre-log}
Suppose that $\Phi$ carries a right-shift  $\sigma:\Phi\rightarrow\Phi$, then the map defined by:
$$\log_{\sigma}\left(t^{\sum_{\phi \in \Phi} \gamma_\phi\mathds{1}_{\phi}}\right)= \log_{\sigma}\left(\prod_{\phi \in \Phi} t_\phi^{\gamma_\phi}\right) = \sum_{\phi \in \Phi} \gamma_\phi t_{\sigma(\phi)}^{-1}$$
induces a pre-logarithm on $k((\Gamma))$. Moreover if $\sigma$ is onto, denote  the \textbf{convex orbit} of any $\phi$ by $\mathcal{C}_\phi=\{\psi\in\Phi\ |\ \exists k\in\mathbb{N},\ \sigma^k(\phi)\preccurlyeq\psi\preccurlyeq \sigma^{-k}(\phi)\}$.
Then the pre-logarithm $\log_\sigma$ corresponds to a series derivation of Hardy type such that:
\begin{description}
\item \emph{For any $\phi_1<\phi_2\in\Phi$, the principal part up to $\mathcal{C}_{\phi_1}$ of  $\theta_{\phi_1}- \theta_{\phi_2}$ is equal to the one of $\displaystyle\sum_{j=1}^\infty \mathds{1}_{\sigma^j(\phi_2)}-\mathds{1}_{\sigma^j(\phi_1)}$,
with in particular $\theta_{\sigma^{k}(\phi)}= \theta_{\phi}-\displaystyle\sum_{j=1}^k\mathds{1}_{\sigma^j(\phi)}$  for any $k\in\mathbb{N}$.}
\end{description}
\end{example}

\section{Exponential-logarithmic series fields}\label{sect:EL-series}

The last but not least differential equation we are interested in is the one for the \emph{exponential}: 
\begin{equation}\label{equ:exp}
 d(y)=d(a)y\Leftrightarrow y=c\exp(a)
\end{equation}
According to \cite{kuhl:exp-pow-series}, there is no hope for defining an exponential and logarithmic structure on a generalized series field. Nevertheless there are several similar ways to obtain non Archimedian exponential-logarithmic fields seen as \emph{subfields} of generalized series fields: see below and \cite{vdd:LE-pow-series, vdh:transs_diff_alg, kuhl:ord-exp}.  See also \cite{kuhm-tressl:EL-LE} for a comparison between LE-series and EL-series. The original idea is due to \cite{dahn:limit-exp-terms,dahn-gorhing} and may be seen as an abstraction of Hardy's construction of log-exp functions \cite{hardy}. In each case, it consists in starting with some initial field of generalized series, and then taking the closure under towering logarithmic and/or exponential extensions. As a main difference, for LE-series and grid-based transseries only grid-based supports are allowed for the series, whereas for EL-series and well-ordered transseries the supports are well-ordered. For instance, the grid-based transseries field is obtained by starting with a field of grid-based series with $\Phi=\N$ corresponding to the various iterates of $\log(x)$, and by applying the exponential closure process described, allowing only grid-based supports for the series \cite[Section 4.3]{vdh:transs_diff_alg}. 

The fields of LE-series and of grid-based transseries are naturally equipped with a series derivation of Hardy type  \cite[Section 5.1]{vdh:transs_diff_alg}. It is also possible to equip certain fields of well-ordered transseries with such a derivation \cite{schm01}. Our aim here is to show that it is also possible to do that for EL-series fields. \\ 

Here we get started directly with a generalized series field $k((\Gamma))$ endowed with a pre-logarithm strong morphism $\log$ and a Hardy type series derivation $d$ as in Section \ref{sect:prelog}. The construction we present here is the construction of the exponential-logarithmic closure of $k((\Gamma))$ as in \cite{kuhl:ord-exp}. It consists in iterating the following exponential extension procedure.

By definition, the pre-logarithm defines an embedding of ordered groups $\log:t^\Gamma\hookrightarrow k((\Gamma_{<0}))$. Keep in mind that this map \emph{cannot} be onto \cite{kuhl:exp-pow-series}. In other words, the exponential of the elements of $k((\Gamma_{<0}))\setminus\log(\Gamma)$ are missing in the left-hand side.  Set the following multiplicative copy of the ordered additive group $\Gamma^\sharp:= k((\Gamma_{<0}))$:
\begin{center}
$t^{\Gamma^\sharp}:=\{t^a\ |\ a\in k((\Gamma_{<0})),\ \mathrm{with}\ t^{\log(t^\alpha))}:=t^\alpha,\ \alpha\in\Gamma\}$
\end{center} 
By construction $t^\Gamma\subset t^{\Gamma^\sharp}$, so the generalized series field $k((\Gamma^\sharp))$ extends $k((\Gamma))$. Define also a pre-logarithm  by:
\begin{center}
$\begin{array}{llcl}
\log^\sharp:&t^{\Gamma^\sharp}&\hookrightarrow&k((\Gamma^\sharp_{<0}))\\
&t^a&\mapsto &a 
\end{array}$
\end{center}
$\log^\sharp$ extends $\log$, and $k((\Gamma_{<0}))$ embeds in $k(( \Gamma^\sharp_{<0}))$. Call $\left(k((\Gamma^\sharp)),\log^\sharp\right)$ the \textbf{exponential extension} of $\left(k((\Gamma)),\log\right)$.

By iterating this exponential extension procedure, we can define the $n$'th exponential extension $\left(k((\Gamma^{\sharp n})),\log^{\sharp n}\right)$, obtaining thus an inductive system of pre-logarithmic fields. The inductive limit is defined as the \textbf{exponential-logarithmic series (EL-series) field} (induced by $\left(k((\Gamma)),\log\right)$ in the sense of \cite{kuhl:ord-exp}, say $\left(k((\Gamma))^{\mathrm{EL}},\log^{\mathrm{EL}}\right)$.

\begin{theorem}[Kuhlmann]
For any generalized series field $k((\Gamma))$ endowed with a pre-logarithm $\log$, the couple $\left(k((\Gamma))^{\mathrm{EL}},\log^{\mathrm{EL}}\right)$ is such that:
\begin{center}
$\log^{\mathrm{EL}}:\left(k((\Gamma))^{\mathrm{EL}}_{>0},.,\leq\right)\rightarrow \left(k((\Gamma))^{\mathrm{EL}},+,\leq\right)$
\end{center}
is an isomorphism of ordered groups, verifying the growth axiom scheme (GA). Therefore it admits a well-defined exponential map $\exp^{\mathrm{EL}}$ as inverse map, and  $k((\Gamma))^{\mathrm{EL}}$ is a non-Archimedian exponential-logarithmic field.
\end{theorem}

Note that $k((\Gamma))^{\mathrm{EL}}$ is a strict subfield of the generalized series field $k((\Gamma^{\mathrm{EL}}))$ where $\Gamma^{\mathrm{EL}}:=\displaystyle\bigcup_{n\in\N} \Gamma^{\sharp n}$. For example, the series $t_\phi^{-1}+t^{-t_\phi}+t^{-t^{t_\phi}}+\cdots$ belongs to $k((\Gamma^{\mathrm{EL}}))\setminus  k((\Gamma))^{\mathrm{EL}}$. Nevertheless, $\Gamma^{\mathrm{EL}}$ is the value group of $k((\Gamma))^{\mathrm{EL}}$ and $k((\Gamma^{\mathrm{EL}}))$, and $k$ is their residue field, so that $k((\Gamma))^{\mathrm{EL}}\subset k((\Gamma^{\mathrm{EL}}))$ is an immediate extension. \\

Considering $k((\Gamma))$ endowed with a Hardy type series derivation $d$, we are interested in extending $d$ to $k((\Gamma))^{\mathrm{EL}}$. In \cite[Theorem 6.2]{matu-kuhlm:hardy-deriv-EL-series}, we showed that this is always feasible:

\begin{theorem}\label{theo:derivEL}
The series derivation of Hardy type $d$ on
$k((\Gamma))$ extends to a series derivation of Hardy type on $k((\Gamma))^{\mathrm{EL}}$, and this extension is uniquely determined.
\end{theorem}

Of course it is understood that the derivations and exponential and logarithmic maps we consider verify the corresponding differential equations (\ref{equ:log}) and (\ref{equ:exp}): such $k((\Gamma))^{\mathrm{EL}}$ is a \emph{non-Archimedian differential exponential-logarithmic field}.

Moreover $k((\Gamma))^{\mathrm{EL}}$ inherits the properties of $k((\Gamma))$ concerning asymptotic integration and integration. Denote as before $\tilde{\theta}=l.u.b.\left\{\theta_{\phi}\ ;\
\phi\in\Phi\right\}$ where $\Phi$ is the rank of the \emph{initial} valued group $\Gamma$. By \cite[Theorem 7.1 and Corollary 7.2]{matu-kuhlm:hardy-deriv-EL-series}, we have that:
\begin{theorem}\label{theo:integ-EL}
 A series $a\in k((\Gamma))^{\rm{ EL }}$ admits an asymptotic integral if and only if $a\nasymp \tilde{\theta}$.

The EL-series field $k((\Gamma))^{\rm{ EL }}$ is closed under integration if and only if $\tilde{\theta}\notin \Gamma^{\rm{ EL}}$.
\end{theorem}

To conclude this section and the article, recall from our introduction that generalized series fields are universal domains for valued fields. In the characteristic zero case, with the EL-series construction, they provide also universal domains for ordered exponential fields. This is brilliantly illustrated by the construction of \emph{surreal numbers }\cite{conway_numb-games,gonshor_surreal,alling:analysis-surreal}. See also \cite{alling_conway-surreal-numb} for a specific emphasis on the generalized series structure and set theoretic topics. These non Archimedian numbers form a \emph{class}, say $\mathbf{NO}$, which can be at the same time viewed as a generalized series field $\R((\mathbf{NO}))$ via the so-called \textbf{Conway normal form} of the surreal numbers
  and  equipped with an \textbf{exponential} \emph{and} a \textbf{logarithmic} functions $\exp$ and $\log=\exp^{-1}$: see \cite[Chapter 5 and 10]{gonshor_surreal}.  
We describe in a recent work \cite{matu-kuhlm:surreel-transseries-EL} a canonical family, say $\Phi$, of representatives of the exp-log equivalence classes of surreal numbers. Let $\Gamma$ be the Hahn group with copies of $\R$ as ribs and $\Phi$ as spine.
With S. Kuhlmann, we conjecture that:
\begin{description}
\item[Differential EL-series conjecture for surreal numbers.] \emph{The class of surreal numbers $\mathrm{\mathbb{NO}}$ may be described as: $$\mathrm{\mathbb{NO}}=\R((\Gamma))^{\rm{ EL }}$$ 
where the exponential and logarithmic functions are Gonshor's $\exp$ and $\log$.\\
Moreover, $\mathrm{\mathbb{NO}}$ can be endowed with a Hardy type series derivation compatible with such EL structure. Consequently, $\mathbf{NO}$ is a universal domain for exp-log differential fields.}
\end{description}

\renewcommand{\refname}{}    
\vspace*{-36pt}              

%
%
%





\end{document}